\documentclass[10pt,a4paper,fleqn]{article}

\usepackage[paper=a4paper,left=25mm,right=25mm,top=25mm,bottom=25mm]{geometry}
\usepackage[utf8]{inputenc}
\usepackage[english]{babel}
\usepackage[T1]{fontenc}
\usepackage{lmodern}
\usepackage{amsmath,amssymb,amsfonts,amsthm,mathtools}
\usepackage{mathrsfs}
\usepackage{dsfont}
\usepackage{bm}
\usepackage{graphicx}
\usepackage{booktabs}
\usepackage{enumitem}
\usepackage[numbers]{natbib}
\usepackage{microtype}
\usepackage{xcolor}
\usepackage{url}
\usepackage[hidelinks]{hyperref}

\allowdisplaybreaks

\newcommand{\E}{\mathbb{E}}

\newcommand{\R}{\mathbb{R}}

\newcommand{\cC}{\mathcal{C}}

\newcommand{\cA}{\mathcal{A}}

\newcommand{\cU}{\mathcal{U}}

\newcommand{\eqd}{\stackrel{\mathrm{d}}{=}}

\newcommand{\1}{\mathds{1}}
\newcommand{\de}{\mathrm{\,d}}

\DeclarePairedDelimiter{\abs}{\lvert}{\rvert}

\newtheorem{theorem}{Theorem}[section]
\newtheorem{proposition}[theorem]{Proposition}
\newtheorem{corollary}[theorem]{Corollary}

\theoremstyle{definition}
\newtheorem{definition}[theorem]{Definition}
\theoremstyle{plain}
\newtheorem{lemma}[theorem]{Lemma}
\theoremstyle{definition}

\theoremstyle{remark}
\newtheorem{remark}[theorem]{Remark}

\title{Copulas farthest from independence \\ in quadratic Wasserstein distance}

\author{Jonathan Ansari\\[0.45em]
\small Department of Mathematics, Paris Lodron University of Salzburg, Austria\\
\small \texttt{jonathan.ansari@plus.ac.at}}

\date{September 19, 2026}

\begin{document}
\maketitle

\begin{abstract}
Let \(\Pi\) denote the independence copula and \(M,W\) the upper and lower Fr\'echet--Hoeffding copulas. \citet{Catalano-Lavenant-2025} conjectured that \(M\) and \(W\) maximize the quadratic Wasserstein distance from \(\Pi\) among all bivariate copulas. We prove this conjecture and characterize all equality cases:
\(\mathcal{W}_2^2(C,\Pi)\leq \frac{1}{10}\), with equality if and only if \(C\in\{M,W\}\). We also determine explicitly the optimal Monge map from \(\Pi\) to \(M\). The proof is constructive and combines the optimal transport from independence to the diagonal, a sharp convex-order inequality for one-Lipschitz functions, the conditional convex order, and a coupling construction based on conditional comonotonicity and the supermodular order. As a consequence, we obtain a normalized Wasserstein-based dependence measure that characterizes independence and attains its maximal value exactly for comonotone and countermonotone dependence.
\end{abstract}



\section{Introduction}\label{sec:intro}

For probability measures with prescribed marginals, the Wasserstein distance from the corresponding product measure provides a natural notion for measuring dependence; see, e.g., \citep{Mordant-Segers-2022,Catalano-Lavenant-2025,DeKeyser-Gijbels-2025}. In the simplest continuous setting, let \(\lambda\) be the uniform probability measure on \([0,1]\) and identify bivariate copulas with the couplings \(\Gamma(\lambda,\lambda)\). \citet[Remark~3]{Catalano-Lavenant-2025} asked which copulas are farthest from the independence copula \(\Pi(u,v) = uv =\lambda\otimes\lambda([0,u]\times [0,v])\) in quadratic Wasserstein distance for the Euclidean ground metric, and conjectured that the diagonal and antidiagonal couplings are extremal. \citet{Schrott-2026} recently proved the corresponding Gaussian problem in arbitrary dimension and explicitly recorded the uniform case as open.

Our main result, Theorem \ref{thm:main}, settles the bivariate uniform conjecture and, in addition, determines all maximizers. 
As a consequence, we define a Wasserstein-based dependence measure that characterizes independent and perfect monotone dependence of two random variables \(X\) and \(Y\); see Corollary \ref{cor:main1}. 
In Corollary \ref{cor_main}, we determine the optimal transport map from \(\Pi\) to \(M\), that we illustrate in Figure \ref{fig:Monge_Pi_M}.

\subsection{Main result}
We denote by \(M(u,v) := \min\{u,v\}\) and \(W(u,v) := \max\{u+v-1,0\}\) the upper and lower Fr\'{e}chet Hoeffding copulas, which distribute mass uniformly on the diagonal \(u=v\) and the anti-diagonal \(u+v=1\), respectively. We write \(\cC\) for the set of bivariate copulas and recall that every copula \(C\) is the distribution function associated with a coupling \(\pi\in\Gamma(\lambda,\lambda)\), that is, \(C(u,v) = \pi([0,u]\times [0,v])\) for all \((u,v)\in [0,1]^2\).
Conversely, every copula \(C\in \cC\) induces a coupling in \(\Gamma(\lambda,\lambda)\) that we denote as \(\pi_C\).
The Wasserstein distance with quadratic Euclidean cost is then defined by 
\begin{align} 
\label{def_W1}
\mathcal{W}_2^2(C,D) &:= \inf_{\gamma\in\Gamma(\pi_C,\pi_D)} \int_{[0,1]^2\times[0,1]^2} \|x-y\|_2^2\,\de\gamma(x,y)\\
\label{def_W2}
&\phantom{:}= \inf_{\substack{(X,Y)\sim C\\(X',Y')\sim D}} \E\bigl[(X-X')^2+(Y-Y')^2\bigr]. 
\end{align} 
In \eqref{def_W1}, the infimum is taken over all couplings of \(\pi_C\) and \(\pi_D\). In \eqref{def_W2}, the infimum is taken over all random vectors \((X,Y,X',Y')\) with \((X,Y)\sim C\) and \((X',Y')\sim D\).

\begin{theorem}[Farthest copulas from independence]\label{thm:main}
For every bivariate copula \(C\), we have
\begin{align}\label{eq:main_intro}
 \mathcal{W}_2^2(C,\Pi)\le \frac1{10},
\end{align}
with equality if and only if \(C\in \{M,W\}\).
Consequently,
\[
 \sup_{C\in\cC}\mathcal{W}_2(C,\Pi) = \mathcal{W}_2(M,\Pi) = \mathcal{W}_2(W,\Pi) =\frac1{\sqrt{10}}.
\]
\end{theorem}

A natural way to quantify
dependence is to measure the Wasserstein distance between the joint
distribution and the product of its marginals; see \citep{Mori-Szekely-2020,Mordant-Segers-2022,Nies-Staudt-Munk-2025,
DeKeyser-Gijbels-2025,Catalano-Lavenant-2025}.
For bivariate copulas this leads to the functional \(C\mapsto\mathcal W_2(C,\Pi)\),
which vanishes if and only if \(C=\Pi\). In order to obtain a
dependence coefficient with range \([0,1]\), the distance has to be
normalized by its maximal possible value over all copulas. 
Theorem~\ref{thm:main} determines this normalization explicitly, which yields the following result.

\begin{corollary}\label{cor:main1}
Let \((X,Y)\) be a bivariate random vector with a continuous distribution function and copula \(C\). Then the functional \(\mathfrak D_{\mathcal W}(X,Y):=\mathfrak D_{\mathcal W}(C)
    :=
    \sqrt{10}\,\mathcal W_2(C,\Pi)\), \(C\in \cC\), 
    satisfies
    \begin{enumerate}[label = (\roman*)]
        \item \(\mathfrak D_{\mathcal W}(X,Y)\in [0,1]\),
        \item \(\mathfrak D_{\mathcal W}(X,Y) = 0\) \(\Longleftrightarrow\) \(X\) and \(Y\) are independent \(\Longleftrightarrow\) \(C = \Pi\),
        \item \(\mathfrak D_{\mathcal W}(X,Y) = 1\) \(\Longleftrightarrow\) \(X\) and \(Y\) are comonotone or countermonotone \(\Longleftrightarrow\) \(C\in \{M,W\}\).
    \end{enumerate}
\end{corollary}

\begin{remark}
\begin{enumerate}[label=\textnormal{(\alph*)}]
\item
Theorem \ref{thm:main} complements the recent Gaussian result of
\citet{Schrott-2026}. For \(n\ge2\) standard Gaussian marginals,
Schrott shows that the couplings farthest in quadratic Wasserstein
distance from the independent Gaussian distribution are precisely the
signed diagonal couplings; in dimension two these reduce to the
monotone and antimonotone Gaussian couplings. His proof relies on
Brenier maps, the Gaussian coarea formula, and Gaussian isoperimetry.
In contrast, the uniform setting considered here has no corresponding
Gaussian isoperimetric structure. Our argument is instead based on the
explicit optimal transport from independence to the diagonal, a sharp
convex-order inequality for one-Lipschitz functions, and a conditional
convex-order comparison. In particular, the two results identify the
same qualitative extremal dependence structures in the Gaussian and
bivariate uniform settings, respectively, but by substantially
different mechanisms.
\item
Corollary~\ref{cor:main1} places the Wasserstein coefficient
\(\mathfrak D_{\mathcal W}\) in the context of classical dependence
measures. For continuous random variables, passing to the rank
transforms yields uniform marginals and hence margin-free dependence
measures; see, e.g., \citet{Nelsen-2006}. Classical concordance
measures such as Kendall's tau and Spearman's rho attain their extrema
at \(M\) and \(W\), but do not characterize independence
\citep{Nelsen-2006}. In contrast, the Schweizer--Wolff measure
vanishes exactly at \(\Pi\) and attains its maximum exactly at \(M\)
and \(W\) \citep{Schweizer-Wolff-1981}, thus sharing the extremal
characterizations in Corollary~\ref{cor:main1}. This differs from
measures of functional dependence such as Chatterjee's rank
correlation \citep{Chatterjee-2021} and the Wasserstein correlation
coefficients of \citet{Wiesel-2022}, which compare conditional distributions with unconditional ones and attain their maximal value more generally under
functional dependence \(Y=f(X)\), without requiring monotonicity.
\end{enumerate}
\end{remark}

\begin{figure}[t]
    \centering
    \includegraphics[width=0.65\textwidth]{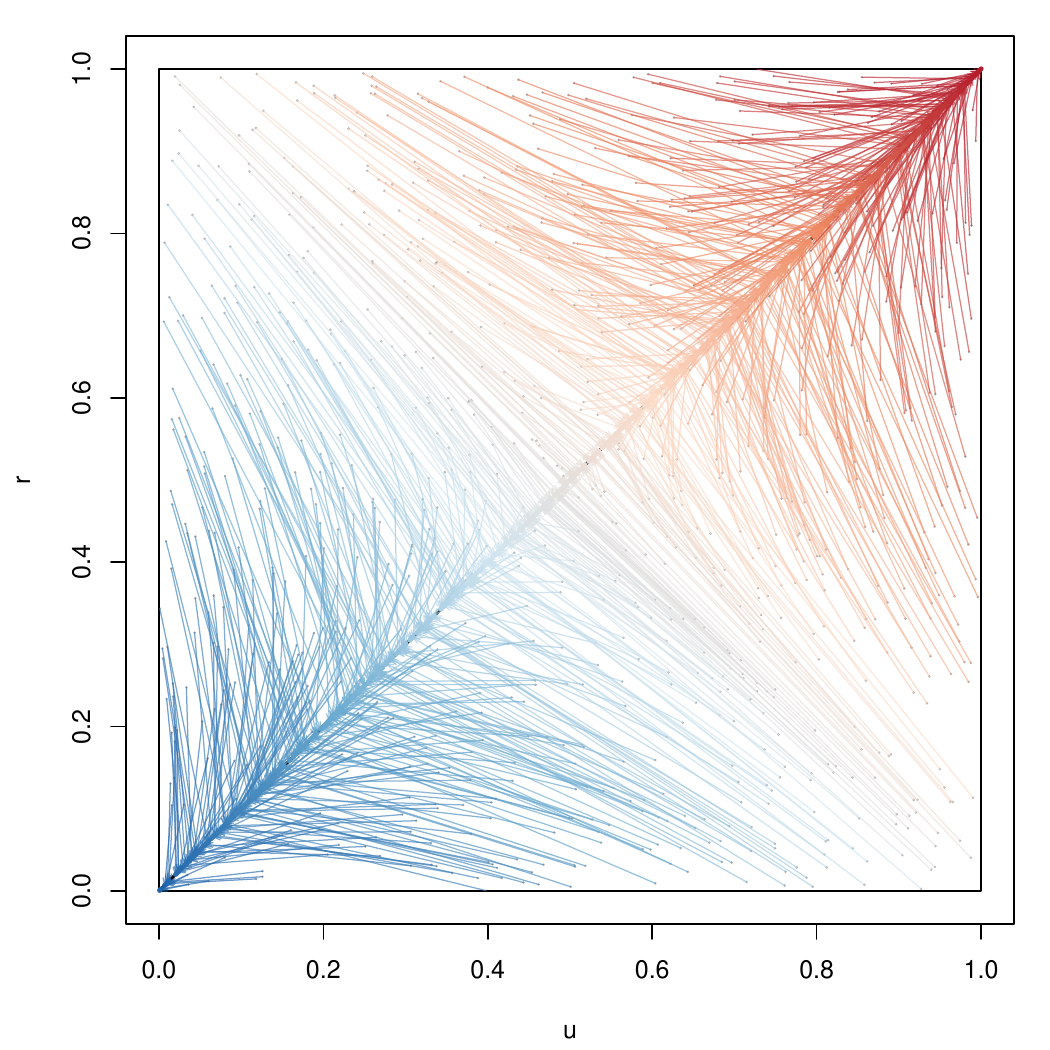}
    \caption{Optimal Monge transport from the independence copula \(\Pi\) to the upper Fr\'echet copula \(M\), induced by the map \(T(u,r)\) in \eqref{eq_Brenier}. The lines map \(1000\) points drawn uniformly from \([0,1]^2\) via \(T\) onto the diagonal.}
    \label{fig:Monge_Pi_M}
\end{figure}

\subsection{Sketch of the proof}
The proof of Theorem \ref{thm:main} is based on a coupling argument that we briefly outline as follows.
The idea is to compare a \emph{suitable} transport from \(\Pi\) to an arbitrary copula \(C\) with the \emph{optimal} transport from \(\Pi\) to the diagonal described by \(M\). Therefore, let \(U,R\sim \cU(0,1)\) be independent and set
\begin{align}\label{def_Xs}
    X:=F_S(U+R),
    \qquad  S:=U+R.
\end{align}
Then \(X\sim \cU(0,1)\), and it is not difficult to see that the diagonal transport
\[
    (U,R)\longmapsto (X,X)
\]
is the optimal transport from \(\Pi\) to \(M\). Indeed, any coupling
between \(\Pi\) and \(M\) can be represented by
\((U,R,Z,Z)\) with \(Z\sim\cU(0,1)\), and its quadratic cost is
\[
    \E[(U-Z)^2+(R-Z)^2]
    =
    \frac43-2\E[(U+R)Z].
\]
Thus minimizing the cost is equivalent to maximizing
\(\E[SZ] = \E[(U+R)Z]\). By the classical rearrangement
inequality, the maximum is attained by the comonotone coupling
\(Z=F_S(S)=X\). Its quadratic cost is
\[
    \E[(U-X)^2+(R-X)^2]
    =
    \frac{1}{10}.
\]
By symmetry, \(\E[(U-X)^2]
    =
    \E[(R-X)^2]
    =
    \frac{1}{20}.
\)

Now let \(C\) be an arbitrary bivariate copula and \(X\) be given by \eqref{def_Xs}. The idea is to construct \(Y\) such that
\[
    (X,Y)\sim C
\]
and \(Y\) is conditionally independent of \((U,R)\) given \(X\). We then construct another uniform random variable \(V\) such that
\[
    (U,V)\sim\Pi
\]
and \(V\) and \(Y\) are comonotone conditionally on \(U\). Hence,
\[
    (U,V)\longmapsto(X,Y)
\]
defines an admissible transport coupling from \(\Pi\) to \(C\). The first coordinate cost is unchanged, i.e.,
\[
    \E[(U-X)^2]=\frac{1}{20},
\]
whereas the main part of the proof shows that the second coordinate cost satisfies
\begin{align}\label{def_inequVY}
    \E[(V-Y)^2]
    \le
    \E[(R-X)^2]
    =
    \frac{1}{20}.
\end{align}
Consequently,
\begin{align}\label{eq_wassbound}
    \mathcal W_2^2(\Pi,C)
    &\le
    \E[(U-X)^2+(V-Y)^2]
    \le
    \E[(U-X)^2+(R-X)^2]
    =
    \frac{1}{10}.
\end{align}

The inequality in \eqref{def_inequVY} is obtained by comparing the conditional distributions of \(Y\) and \(X\) given \(U\). 
The diagonal transport induces the functions
\[
    g_v(t)
    =
    F_{X|U=t}(v)
    =
    [(F_S^{-1}(v)-t)\vee 0 ]\wedge 1;
\]
see Lemma \ref{lem:g_v}.
In contrast, for the target copula \(C\), the above construction yields
conditional distribution functions
\[
    h_v(t):=F_{Y|U=t}(v)
           =(Kf_v)(t),
    \qquad
    f_v(x):=\mathbb P(Y\le v\mid X=x),
\]
where \(K\) is a smoothing kernel defined via \(F_S\) in \eqref{eq:K_def}.
We show that every such \(h_v\) is \(1\)-Lipschitz and satisfies the convex order
\begin{align}\label{eq_conv_ord_hv_gv}
    h_v(U)\le_{\mathrm{cx}}g_v(U) \qquad \text{for all } v\in [0,1].
\end{align}
In fact, \(g_v(U)\) is maximal in convex order among all \(h(U)\) 
with \(h:[0,1]\to[0,1]\) \(1\)-Lipschitz and
\(\int_0^1h(t)\,\de t = v\); see Lemma \ref{lem:Lip_cx}.
The comparison in \eqref{eq_conv_ord_hv_gv} defines the conditional convex order
\[
    (Y,U)\preccurlyeq_{\mathrm{ccx}}(X,U),
\]
a dependence order recently introduced in \cite{Ansari-Fuchs-2025}.
The conditional convex order yields a comparison of the conditionally comonotone random vectors considered above. More precisely, we obtain from its characterization via the supermodular order in Proposition \ref{lem_2_8} the comparison
\begin{align}\label{eq_sm_vyrx}
    (Y,V) \geq_{sm} (X,R);
\end{align}
see Corollary \ref{cor_transcost}.
Here we use that \(Y\) and \(V\), as well as \(X\) and \(R\), are comonotone conditionally on \(U\), while both \(V\) and \(R\) are independent of \(U\); see Lemma \ref{lem_coupling}.
The supermodular comparison in \eqref{eq_sm_vyrx}
yields the desired quadratic-cost inequality
\[
    \E[(V-Y)^2]
    \le
    \E[(R-X)^2];
\]
see again Corollary \ref{cor_transcost}.
Equality in the Wasserstein bound \eqref{eq_wassbound} forces equality in the corresponding comparison of the conditionally comonotone couplings. As shown in Lemma~\ref{lem:smoothing_rigidity} and in the proof of Theorem~\ref{thm:main}, this equality case yields \(Y=X\) or \(Y=1-X\) almost surely, giving precisely the two extremal copulas \(M\) and \(W\).

\subsection{The optimal Monge transport from
\texorpdfstring{\(\Pi\)}{Pi} to \texorpdfstring{\(M\)}{M}}

For quadratic transport costs, it is well known that, if the source measure is absolutely continuous with respect to the Lebesgue measure, the optimal transport is induced by a Monge map given by the gradient of a convex function \(\Phi\); this map is commonly referred to as the Brenier map; see \citet{Rueschendorf-Rachev-1990} and \citet{Brenier-1991}. In the present setting, the Brenier map from \(\Pi\) to \(M\) can be determined explicitly; see Figure~\ref{fig:Monge_Pi_M} for an illustration.

\begin{corollary}[Optimal Monge transport]\label{cor_main}
    The Brenier map from \(\Pi\) to \(M\) is the function \(T\colon [0,1]^2 \to [0,1]^2\) given by
    \begin{align}\label{eq_Brenier}
        T(u,r) = \nabla \Phi(u,r) = \bigl(F_S(u+r),F_S(u+r)\bigr), 
    \end{align}
    where the potential \(\Phi\colon [0,1]^2\to \R\) and the distribution function \(F_S\colon [0,2]\to [0,1]\) are given by
\begin{align}
\Phi(u,r)
&=
\begin{cases}
\dfrac{(u+r)^3}{6},
& u+r\le 1,\\[2mm]
-\dfrac{(u+r)^3}{6}
+(u+r)^2-(u+r)+\dfrac13,
& u+r>1,
\end{cases}
\label{eq:Brenier_potential}
\\[3mm]
F_S(s)
&=
\begin{cases}
\dfrac{s^2}{2},
& 0\le s\le 1,\\[2mm]
1-\dfrac{(2-s)^2}{2},
& 1<s\le 2.
\end{cases}
\label{eq:FS}
\end{align}
\end{corollary}

\subsection{Organization of the paper}

The rest of the paper is organized as follows.
Section~\ref{sec:prelim} collects the required preliminaries on copulas and conditional distributions, the supermodular order, and the conditional convex order.
Section~\ref{sec:diagonal} determines the optimal transport from independence to the diagonal and analyzes the corresponding conditional distributions.
Section~\ref{sec:convex} establishes the sharp convex-order inequality for \(1\)-Lipschitz functions and introduces the associated Markov smoothing operator.
Section~\ref{sec:couplings} constructs, for an arbitrary target copula \(C\), a suitable transport coupling from \(\Pi\) to \(C\) and derives the key transport-cost comparison.
Finally, Section~\ref{sec:main} proves Theorem~\ref{thm:main}, including the characterization of the equality cases, as well as Corollary~\ref{cor_main} on the optimal Monge transport from \(\Pi\) to \(M\).

\section{Preliminaries}\label{sec:prelim}

Throughout, \(\lambda\) denotes the Lebesgue probability measure on
\([0,1]\), and \(\cU(0,1)\) the corresponding uniform distribution.
All random variables are defined on a common probability space
\((\Omega,\cA,\mathbb P)\), which we assume to be sufficiently rich.
For random variables \(X\) and \(Y\), we write \(X\eqd Y\) for equality
in distribution.
For probability measures \(\mu\) and \(\nu\), we denote by
\(\Gamma(\mu,\nu)\) the set of all couplings of \(\mu\) and \(\nu\),
that is, the set of distributions with first and second marginal
\(\mu\) and \(\nu\), respectively.
For \(\pi\in\Gamma(\lambda,\lambda)\), its distribution function
\[
    C_\pi(u,v):=\pi([0,u]\times[0,v]),
    \qquad (u,v)\in[0,1]^2,
\]
is a \emph{copula}, i.e., a bivariate distribution function with
uniform marginals. In particular,
\[
    C_\pi(u,1)=u=C_\pi(1,u),
    \qquad u\in[0,1].
\]
Conversely, every bivariate copula \(C\) uniquely determines a
probability measure \(\pi_C\in\Gamma(\lambda,\lambda)\). We denote by
\(\cC\) the set of all bivariate copulas and write \((X,Y)\sim C\)
whenever the distribution of \((X,Y)\) is \(\pi_C\). Recall the definition of the Wasserstein distance \(\mathcal{W}_2(C,D)\) in \eqref{def_W1}.
Further recall the copulas
\begin{align}\label{eq:M_Pi_W}
    \Pi(u,v):=uv,\qquad
    M(u,v):=\min\{u,v\},\qquad
    W(u,v):=\max\{u+v-1,0\}.
\end{align}
Their associated probability measures are the laws of
\((U,V)\), \((U,U)\), and \((U,1-U)\), respectively, where
\(U,V\sim\cU(0,1)\) are independent. 
Hence, \(\Pi\) models independence, whereas \(M\) and \(W\) correspond to
comonotonicity and countermonotonicity, respectively. Recall that random
variables \(X\) and \(Y\) are called \emph{comonotone} if there exist a
random variable \(Z\) and nondecreasing functions \(f,g\) such that
\((X,Y)\eqd(f(Z),g(Z))\),
and \emph{countermonotone} if there exist a random variable \(Z\), a
nondecreasing function \(f\), and a nonincreasing function \(g\) such that
\((X,Y)\eqd(f(Z),g(Z)).\)

\subsection{Copula derivatives}

Let \(C\in \cC\) be a copula and let \((V_1,V_2)\sim C\) be a bivariate random vector with distribution \(\pi_C\). Then, the conditional distribution function of \(V_2 \mid V_1=v_1\) can be represented by the first partial derivative of \(C\). More precisely, for fixed \(v_2\), we have
\begin{align*}
    F_{V_2|V_1=v_1}(v_2) = \partial_1 C(v_1,v_2) \qquad \text{for } \lambda\text{-almost all } v_1\in [0,1],
\end{align*}
where the exceptional null set may depend on \(v_2\). Note that \(x\mapsto \partial_1C(v_1,x)\) is increasing but not necessarily right-continuous. However, by the existence of regular conditional distributions, one can construct a Markov kernel \(K_C\) associated with \(C\) such that, for any \(v_2\),
\begin{align*}
    F_{C,v_1}(v_2):= K_C(v_1,[0,v_2]) = \partial_1 C(v_1,v_2) \qquad \text{for } \lambda\text{-almost all } v_1\in [0,1];
\end{align*}
see e.g. \cite{Durante-Sempi-2016} for details.
Then, \(x\mapsto F_{C,v_1}(x)\) is a distribution function for all \(v_1\). We denote by
\begin{align}
    F_{C,v_1}^{-1}(t) := \inf\{x\in [0,1]\mid F_{C,v_1}(x)\geq t\}, \qquad t\in (0,1),
\end{align}
the (left-continuous) generalized inverse of \(F_{C,v_1}\). 

\begin{lemma}[Conditional quantile construction]
\label{lem:conditional_quantile_construction}
For \(C\in\cC\), let \(X,\Lambda\sim\cU(0,1)\) be independent, and define
\[
    Y:=F_{C,X}^{-1}(\Lambda).
\]
Then we have \((X,Y)\sim C\).
\end{lemma}

\begin{proof}
Conditionally on \(X=x\), the variable \(Y=F_{C,x}^{-1}(\Lambda)\) has
distribution function \(F_{C,x}\). Hence, for \(u,v\in[0,1]\),
\begin{align*}
    \mathbb P(X\le u,Y\le v)
    &=\int_0^u
      \mathbb P(Y\le v\mid X=x)\,\de x
    =\int_0^u F_{C,x}(v)\,\de x
     =C(u,v),
\end{align*}
where the last identity follows from disintegration theorem.
\end{proof}

\subsection{Supermodular order}

For bivariate random vectors \((Y,Z)\) and \((Y',Z')\) with \(\cU(0,1)\)-marginals, the \emph{supermodular order} \((Y,Z)\leq_{sm} (Y',Z')\) is defined by 
\begin{align*}
     \E f(Y,Z)\leq \E f(Y',Z') \quad \text{for all bounded supermodular functions } f\colon [0,1]^2\to \R;
\end{align*}
see e.g. \citet{Mueller-Stoyan-2002}. Recall that a function \(f\) is \emph{supermodular} if \(f(y) + f(z)\leq f(y\wedge z) + f(y\vee z)\) for all \(y,z\in [0,1]^2\), where \(\wedge\) and \(\vee\) denote the componentwise minimum and maximum, respectively. 
The supermodular order is well known in optimal transport theory since comonotone couplings are maximal elements and thus solve, for convex cost functions, optimal transport problems \emph{on the real line}; see, e.g., \citet[Theorem~3.1.2]{Rachev-Rueschendorf-1998}.

For optimal transport problems between distributions on \(\R^2\), the situation is much more challenging.
For the identification of the copulas farthest from independence, we will need
the following lemma. It shows that two supermodularly ordered bivariate random
vectors with uniform marginals and identical quadratic costs must have
the same distribution.

\begin{lemma}[Quadratic cost and supermodular order]
\label{lem:J_lower_orthant}
Let \(Y,Z,Y',Z'\sim\cU(0,1)\) and suppose that
\((Y,Z)\geq_{\mathrm{sm}}(Y',Z')\)
Then, we have
\[
    \E[(Y-Z)^2]
    \le
    \E[(Y'-Z')^2].
\]
with equality if and only if \((Y,Z)\eqd(Y',Z')\).
\end{lemma}

\begin{proof}
Let \(C\) and \(D\) denote the copulas of \((Y,Z)\) and
\((Y',Z')\), respectively. Then \((Y,Z)\geq_{sm} (Y',Z')\) is equivalent (in the bivariate case) to the pointwise comparison
\begin{align}\label{eq_CD}
    C(u,v)\ge D(u,v)
    \qquad\text{for all }(u,v)\in[0,1]^2;
\end{align}
see \citet[Theorem~2.5]{Mueller-Scarsini-2000}. Using \eqref{eq_CD} and uniform marginals, we obtain
\begin{align*}
    \E[(Y-Z)^2] &= \E[Y^2] + \E[Z^2] - 2 \E[YZ] \\
    &= \frac 1 3 + \frac 1 3 - 2 \int_0^1\int_0^1 \mathbb P(Y>u,Z>v)\de u \de v \\
    &= \frac 2 3 - 2 \int_0^1 \int_0^1 [1 - u - v + C(u,v)] \de u \de v \\
    &= \frac 2 3 - 2 \int_0^1 \int_0^1 C(u,v) \de u \de v \\
    &\leq \frac 2 3 - 2 \int_0^1 \int_0^1 D(u,v) \de u \de v  = \ldots = \E[(Y'-Z')^2]
\end{align*}
with equality if and only if \(C=D\) \(\lambda^2\)-almost surely on \([0,1]^2\). Since copulas are continuous, the statement follows. 
\end{proof}

\subsection{Conditional convex order}
For bounded random variables \(A,B\), the \emph{convex order} \(A\le_{\mathrm{cx}}B\) is defined by \(\E\varphi(A)\le\E\varphi(B)\) for all convex functions \(\varphi\colon \R \to \R\). 
The conditional convex order, recently introduced in \cite{Ansari-Fuchs-2025}, compares conditional probabilities in convex order. Here, we consider the specific case of bivariate \(\cU(0,1)\)-distributed random variables.
\begin{definition}[Conditional convex order]
    For random variables \(Y,Y',Z,Z'\sim \cU(0,1)\), the \emph{conditional convex order} \((Y,Z)\preccurlyeq_{\mathrm{ccx}}(Y',Z')\) is defined by
    \begin{align}\label{eq:ccx_special}
    F_{Y|Z}(v)\le_{\mathrm{cx}}F_{Y'|Z'}(v)
 \quad \text{for all } v\in[0,1].
\end{align}
\end{definition}

\begin{remark}
The conditional convex order differs fundamentally from the
supermodular order in its interpretation of dependence. The
supermodular order compares the degree and direction of concordance
within a fixed Fr\'echet class: in the bivariate case, its extremal
elements are the countermonotone and comonotone couplings \(W\) and
\(M\), respectively. In contrast, the conditional convex order is an
ordering of the strength of functional dependence. Independence forms
its minimal element, whereas perfectly functionally dependent pairs form its
maximal class; see \cite{Ansari-Fuchs-2025} for details. Thus, it does not distinguish positive from negative association, but rather compares how strongly the first component is
determined by the second.
\end{remark}

The following result will be key to the proof of Theorem~\ref{thm:main}.
It provides a characterization of the conditional convex order in terms of
a supermodular comparison of conditionally comonotone random vectors.

\begin{proposition}[Characterization of \(\preccurlyeq_{ccx}\)]\label{lem_2_8}
    For random variables \(Y,Y',Z,Z'\sim \cU(0,1)\), we have
    \begin{align}
        (Y,Z)\preccurlyeq_{ccx} (Y',Z') \qquad \Longleftrightarrow \qquad (F_{Y|Z}^{-1}(\xi),\xi) \geq_{sm} (F_{Y'|Z'}^{-1}(\xi'),\xi')
    \end{align}
    for \(\xi\sim \cU(0,1)\) independent of \(Z\) and for \(\xi'\sim \cU(0,1)\) independent of \(Z'\).
\end{proposition}

\begin{proof}
\((Y,Z)\preccurlyeq_{ccx} (Y',Z')\) is equivalent to \((F_{Y|Z}^{-1}(\xi),\xi) \geq_{c} (F_{Y'|Z'}^{-1}(\xi),\xi)\), where \(\geq_c\) denote the concordance order; see \citet{Ansari-Fuchs-2025}. However, for bivariate random vectors, the concordance order is equivalent to the supermodular order; see \cite[Theorem 2.5]{Mueller-Scarsini-2000}.
\end{proof}

\section{The optimal transport from independence to the diagonal}\label{sec:diagonal}

For independent random variables \(U,R\sim \cU(0,1)\), define their sum \(S:=U+R\).
The distribution function of \(S\) is
\begin{align}\label{eq:triangular_cdf}
    F_S(s)=
    \begin{cases}
        \dfrac{s^2}{2}, & 0\le s\le1,\\[1mm]
        1-\dfrac{(2-s)^2}{2}, & 1\le s\le2.
    \end{cases}
\end{align}
It is continuous and strictly increasing on \([0,2]\). Hence, the probability integral transform of \(S\) is uniform on \([0,1]\), i.e.,
\begin{align*}
    F_S(U+R)\sim \cU(0,1).
\end{align*}
In the following proposition, we establish an optimal transport map from \(\Pi\) to \(M\) and determine its cost.

\begin{proposition}[Optimal transport to the diagonal]\label{lem:diag_transport}
The coupling
\[
    (U,R)\longmapsto (X,X),
    \qquad X=F_S(U+R),
\]
is optimal from \(\Pi\) to \(M\) for the quadratic Euclidean cost. Moreover,
\begin{align}\label{eq:W2_Pi_M}
    \mathcal{W}_2^2(\Pi,M)
    =
    \mathcal{W}_2^2(\Pi,W)
    =
    \frac{1}{10}.
\end{align}
\end{proposition}

\begin{proof}
Every coupling between \(\Pi\) and \(M\) can be represented by a triple
\((U,R,Z)\) such that \((U,R)\sim\Pi\), \(Z\sim\cU(0,1)\), where the transport is from the initial point \((U,R)\) to the target
point \((Z,Z)\). Then the expected quadratic transport cost is
\begin{align}\label{eq:diag_cost_expand}
\begin{split}
    \E\bigl[(U-Z)^2+(R-Z)^2\bigr]
    &=\E[U^2+R^2+2Z^2]-2\E[(U+R)Z]
    =\frac{4}{3}-2\E[SZ].
\end{split}
\end{align}
Consequently, minimizing the transport cost is equivalent to maximizing
\(\E[SZ] = \E[(U+R)Z]\) over all couplings of \(S\) with a uniform random variable \(Z\).
By the classical maximal-correlation property of the comonotone coupling, this expectation is maximized when \(S\) and \(Z\) are coupled comonotonically; see,
for instance, \citet[Remark 3.25]{Rueschendorf-2013}. Since \(F_S(S)\) is uniform and
\(F_S\) is strictly increasing on \([0,2]\), this coupling is
\begin{align}\label{eq_opt_coupling}
    Z=F_S(S)=X.
\end{align}
Consequently, an optimal Monge transport from \(\Pi\) to \(M\) is given by
\[ T\colon [0,1]^2\to[0,1]^2,
\qquad
T(u,r)
=
\bigl(F_S(u+r),F_S(u+r)\bigr).
\]
To determine the optimal transport cost, we observe that the density of \(S\) equals \(s\) on \([0,1]\) and \(2-s\) on \([1,2]\).
Hence
\begin{align*}
    \E[SF_S(S)]
    &=
    \int_0^1 s\,\frac{s^2}{2}\,s\,\de s
    +
    \int_1^2
    s\left(1-\frac{(2-s)^2}{2}\right)(2-s)\,\de s
    =\frac{37}{60}.
\end{align*}
Substitution into \eqref{eq:diag_cost_expand} gives
\[
    \mathcal{W}_2^2(\Pi,M)
    =
    \frac{4}{3}-2\cdot\frac{37}{60}
    =
    \frac{1}{10}.
\]
Finally, a simple reflection argument shows that the lower Fr\'{e}chet copula, supported on the anti-diagonal, is also optimal. This gives \(\mathcal{W}_2(\Pi,W)=\mathcal{W}_2(\Pi,M)\),
which completes the proof.
\end{proof}

\begin{corollary}\label{cor:half_cost}
For the optimal coupling in \eqref{eq_opt_coupling}, we have
\begin{align}\label{eq:half_cost}
    \E[(U-X)^2]
    =
    \E[(R-X)^2]
    =
    \frac{1}{20}.
\end{align}
\end{corollary}

\begin{proof}
The construction of \(X\) is symmetric in \(U\) and \(R\). Hence, the statement follows from \eqref{eq:W2_Pi_M} and \eqref{eq:diag_cost_expand}.
\end{proof}

We shall also need the quantile function of \(S\), given by
\begin{align}\label{eq:q_def}
    q_S(v):=F_S^{-1}(v)
    =\begin{cases}
       \sqrt{2v}, & 0\le v\le \frac12,\\[1mm]
       2-\sqrt{2(1-v)}, & \frac12 < v\le1.
     \end{cases}
\end{align}
For \(z\in \R\), we write \([z]_0^1:=\min\{1,\max\{0,z\}\}\)

\begin{lemma}[Conditional distributions of the diagonal transport]\label{lem:g_v}
For \(v,u\in[0,1]\), the conditional distribution function of \(X\) given \(U=t\) is given by
\begin{align}\label{eq:g_v}
    g_v(t)
    := F_{X|U=t}(v) 
    =[q_S(v)-t]_0^1.
\end{align}
In particular, \(g_v\colon[0,1]\to[0,1]\) is decreasing, \(1\)-Lipschitz, and satisfies \(\int_0^1 g_v(t)\,\de t=v\).
\end{lemma}

\begin{proof}
Since \(F_S\) is strictly increasing, conditionally on \(U=t\), we have
\[
    \{X\le v\}
    =\{F_S(t+R)\le v\}
    =\{R\le q_S(v)-t\}.
\]
Since \(R\) is uniform on \((0,1)\) and independent of \(U\), this gives \eqref{eq:g_v}. The monotonicity and Lipschitz property in \(t\) are immediate. Finally, the integral identity is a consequence of disintegration and \(X\sim \cU(0,1)\).
\end{proof}

\section{A sharp convex-order inequality and Markov smoothing}\label{sec:convex}

A key ingredient to the proof of Theorem \ref{thm:main} is the following extremal property of the function \(g_v\) in \eqref{eq:g_v}.

\begin{lemma}[Sharp \(1\)-Lipschitz rearrangement inequality]\label{lem:Lip_cx}
For \(v\in[0,1]\), let \(h\colon[0,1]\to[0,1]\) be Lipschitz-continuous with Lipschitz constant \(1\) and \(\int_0^1h(t)\,\de t=v\).
Then, we have 
\begin{align*}
h(U)\leq_{\mathrm{cx}}g_v(U)
\end{align*}
for \(U\sim\cU(0,1)\).
\end{lemma}

\begin{proof}
For a measurable function \(f\colon [0,1]\to [0,1]\), we denote by \(f^\downarrow\) its decreasing rearrangement, i.e., its (essentially with respect to \(\lambda\)) uniquely determined decreasing function \(f^\downarrow\) such that \(\lambda(f\leq t ) = \lambda(f^\downarrow\leq t)\) for all \(t\in [0,1]\).
Then \(f(U)\eqd f^\downarrow(U)\) for \(U\sim U(0,1)\). 
It is well known that the decreasing rearrangement does not increase the
Lipschitz constant; see, e.g., \citet[Corollary 2.2]{Yanagihara-1993}. Hence, \(h\) satisfies
\begin{align}\label{eq:rearr_Lip}
    |h^\downarrow(u) - h^\downarrow(v)| \leq |u-v| \qquad \text{for all } u,v\in [0,1].
\end{align}
Now, we compare \(h^\downarrow\) with \(g_v\). On \(\{t\colon g_v(t)=1\}\), we have \(h^\downarrow-g_v\le0\). 
On the interval \(\{0<g_v<1\}\), the function \(g_v\) has slope \(-1\), while \eqref{eq:rearr_Lip} implies that \(h^\downarrow(u)+u\) is nondecreasing. Hence \(h^\downarrow-g_v\) is nondecreasing on that interval. Finally, on \(\{g_v=0\}\), one has \(h^\downarrow-g_v\ge0\). Therefore \(h^\downarrow-g_v\) has at most one sign change, necessarily from nonpositive to nonnegative.

By assumption and Lemma~\ref{lem:g_v}, the two functions \(h\) and \(g_v\) have the same integral. The sign change property consequently yields
\begin{align}\label{eq:majorization_integral}
    \int_0^s h^\downarrow(t)\,\de t
    \le
    \int_0^s g_v(t)\,\de t,
    \qquad 0\le s\le1,
\end{align}
with equality at \(s=1\). By the Hardy-Littlewood-Polya theorem \cite[Theorem 3.21]{Rueschendorf-2013}, the majorization in \eqref{eq:majorization_integral} is equivalent to the inequality
\begin{align*}
    \E\varphi(h^\downarrow(U)) = \int_0^1 \varphi(h^\downarrow(t))\de t \leq \int_0^1 \varphi(g_v(t))\de t = \E \varphi(g_v(U)) \quad \text{f.a. convex }\varphi,
\end{align*}
where \(U\sim \cU(0,1)\). This implies the statement.
\end{proof}

\begin{remark} The clipped-affine extremizer \(g_v\) in \eqref{eq:g_v} is the uniform analogue of the isoperimetric extremizers used in the Gaussian argument of \citet{Schrott-2026}. \end{remark}

We next consider a smoothing operation generated by the diagonal transport. To this end, we define for measurable \(f\colon[0,1]\to[0,1]\) the function
\begin{align}\label{eq:K_def}
    (Kf)(t)
    :=\int_0^1 f\bigl(F_S(t+r)\bigr)\,\de r = \E[f(X)\mid U=t],
    \qquad t\in[0,1].
\end{align}
The following lemma shows that the operator \(K\) maps \([0,1]\)-valued measurable functions into \(1\)-Lipschitz functions.

\begin{lemma}[Markov smoothing]\label{lem:Markov_smoothing}
For \(v\in [0,1]\), let \(f\colon[0,1]\to[0,1]\) be measurable with \(\int_0^1f(t)\,\de t=v\).
Then \(Kf\) in \eqref{eq:K_def} is \(1\)-Lipschitz and satisfies
\begin{align}\label{eq:K_mean}
    \int_0^1(Kf)(t)\,\de t=v.
\end{align}
\end{lemma}

\begin{proof}
For \(0\le t<t'\le1\), the substitution \(s=t+r\) gives
\begin{align}
    \label{int112} (Kf)(t')-(Kf)(t)
    &=\int_{t'}^{t'+1}f(F_S(s))\,\de s
      -\int_t^{t+1}f(F_S(s))\,\de s\\
    \label{int223}&=\int_{t+1}^{t'+1}f(F_S(s))\,\de s
      -\int_t^{t'}f(F_S(s))\,\de s.
\end{align}
Each integral in \eqref{int223} belongs to \([0,t'-t]\), because \(0\le f\le1\). Hence
\[
    \abs{(Kf)(t')-(Kf)(t)}\le t'-t.
\]
Thus \(Kf\) is one-Lipschitz.
Further, by Fubini's theorem and \(X=F_S(U+R)\sim\cU(0,1)\), we have
\[
    \int_0^1(Kf)(t)\,\de t
    =\E[f(X)]
    =\int_0^1f(x)\,\de x
    =v.
\]
This proves \eqref{eq:K_mean}.
\end{proof}

\section{Construction of suitable transport couplings}
\label{sec:couplings}

We now use the results of the previous sections to construct, for an
arbitrary target copula \(C\), a transport coupling from \(\Pi\) to \(C\)
whose quadratic cost does not exceed the optimal transport cost from
\(\Pi\) to \(M\).

To this end, recall that \(U,R\sim\cU(0,1)\) are independent and that
the diagonal transport
\begin{align}\label{eq_constrX}
    (U,R)\mapsto (X,X),
    \qquad
    X:=F_S(U+R),
    \qquad
    S=U+R,
\end{align}
is optimal from \(\Pi\) to \(M\).

Now let \(C\in\cC\) be arbitrary and let
\(\Lambda,\Theta\sim\cU(0,1)\) be independent of each other and jointly
independent of \((U,R)\). Define
\begin{align}
    Y&:=F_{C,X}^{-1}(\Lambda),
    \label{eq:Y_cond_quantile}\\
    V&:=H_U^-(Y)
       +\Theta\bigl(H_U(Y)-H_U^-(Y)\bigr),
    \label{eq:V_cond_transform}
\end{align}
where 
\begin{align}\label{def_H_u}
    (H_u^-(y):=\mathbb P(Y<y\mid U=u) \qquad \text{and} \qquad H_u(y):=\mathbb P(Y\le y\mid U=u).
\end{align}
The random vector \((U,V,X,Y)\) has the following properties.

\begin{lemma}[Construction of a suitable transport coupling]
\label{lem_coupling}
For the random vector \((U,V,X,Y)\) constructed above, we have
\begin{enumerate}[label=(\roman*)]
    \item \label{lem_coupling1} \((U,V)\sim\Pi\),
    \item \label{lem_coupling2} \((X,Y)\sim C\),
    \item \label{lem_coupling3} \(Y\) and \((U,R)\) are conditionally
          independent given \(X\),
    \item \label{lem_coupling4} \(Y\) and \(V\) are comonotone
          conditionally on \(U\),
    \item \label{lem_coupling5} \(X = F_{X|U}^{-1}(R)\) and \(Y = F_{Y|U}^{-1}(V)\) \(\mathbb{P}\)-almost surely.
\end{enumerate}
In particular, \((U,V)\mapsto(X,Y)\)
defines a transport coupling from \(\Pi\) to \(C\).
\end{lemma}

\begin{proof}
Since \(X\) is measurable with respect to \((U,R)\) and
\(\Lambda\) is independent of \((U,R)\), the random variables
\(X\) and \(\Lambda\) are independent. Hence,
Lemma~\ref{lem:conditional_quantile_construction} yields
\[
    (X,Y)
    =
    \bigl(X,F_{C,X}^{-1}(\Lambda)\bigr)
    \sim C,
\]
which proves \ref{lem_coupling2}.

Since \(X\) is measurable with respect to \(\sigma(U,R)\) and
\(\Lambda\) is independent of \((U,R)\), we have
\[
    \Lambda\perp (U,R)\mid X.
\]
Indeed, for bounded measurable \(f\) and \(g\), we have
\begin{align*}
\E[f(\Lambda)g(U,R)\mid X]
&=
\E\!\left[
    \E\!\left[f(\Lambda)g(U,R)\mid U,R\right]
    \,\middle|\, X
\right] 
=
\E\!\left[
    g(U,R)\,
    \E\!\left[f(\Lambda)\mid U,R\right]
    \,\middle|\, X
\right] \\
&=
\E\!\left[
    g(U,R)\,
    \E[f(\Lambda)]
    \,\middle|\, X
\right] 
=
\E[f(\Lambda)]\,
\E[g(U,R)\mid X] \\
&=
\E[f(\Lambda)\mid X]\,
\E[g(U,R)\mid X],
\end{align*}
where we use for the first equality that \(X\) is a function of \((U,R)\). For the third equality, we use that \(\Lambda\) is independent of \((U,R)\), and thus also of \(X\), which yields the last equality. 
Now, since \(Y=F_{C,X}^{-1}(\Lambda)\) is a measurable function of
\((X,\Lambda)\), we obtain
\[
    Y\perp (U,R)\mid X,
\]
which proves \ref{lem_coupling3}.

Next, by the properties of the randomized distributional transform in \citet[Section~3]{Rueschendorf-2009}, \(V\) defined in \eqref{eq:V_cond_transform}
satisfies
\[
    (V\mid U=u)\sim\cU(0,1)
\]
for \(\lambda\)-almost all \(u\in[0,1]\). Since this conditional
distribution does not depend on \(u\), \(V\) is independent of \(U\).
This proves \ref{lem_coupling1}.

Denote the generalized inverse of the conditional distribution function \(H_u\) in \eqref{def_H_u} by \(H_u^{-1}(v) := \inf\{y\in [0,1]\colon H_u(y)\geq v\}\) for \(v\in [0,1]\). Then the randomized conditional quantile transform satisfies
\begin{align*}
    Y = H_U^{-1}(V) \qquad \mathbb{P}\text{-almost surely};
\end{align*}
see \citet[Section 3]{Rueschendorf-2009}.
Since, for every fixed \(u\), the generalized inverse
\(v\mapsto H_u^{-1}(v)\) is nondecreasing, \(Y\) and \(V\)
are comonotone conditionally on \(U\), proving
\ref{lem_coupling4}.

To prove \ref{lem_coupling5}, first fix \(U=u\). Since
\(X=F_S(u+R)\) with \(R\sim\cU(0,1)\), and the map
\(r\longmapsto F_S(u+r)\)
is strictly increasing, it is the conditional quantile function of
\(X\) given \(U=u\). Hence
\[
    X=F_{X\mid U}^{-1}(R)
    \qquad \mathbb P\text{-almost surely}.
\]
Moreover, by the randomized distributional transform used in
\eqref{eq:V_cond_transform},
\[
    Y=\tilde H_U^{-1}(V)
    \qquad \mathbb P\text{-almost surely},
\]
where \(\tilde H_u=F_{Y\mid U=u}\). Therefore,
\(Y=F_{Y\mid U}^{-1}(V)\) \(\mathbb P\text{-almost surely}\);
see \citet[Section~3]{Rueschendorf-2009}.

Finally, \ref{lem_coupling1} and \ref{lem_coupling2} show that the
source vector \((U,V)\) has copula \(\Pi\), whereas the target vector
\((X,Y)\) has copula \(C\). Hence their joint law defines a transport
coupling from \(\Pi\) to \(C\).
\end{proof}

\begin{proposition}[Conditional convex-order comparison]
\label{lem:ccx_comparison}
For the coupling \((U,V,X,Y)\) constructed in \eqref{eq_constrX}--\eqref{eq:V_cond_transform}, we have
\begin{align}\label{eq_lem:ccx_comparison}
    (Y,U)\preccurlyeq_{\mathrm{ccx}}(X,U).
\end{align}
\end{proposition}

\begin{proof}
For \(v\in[0,1]\), define the function
\begin{align}\label{def_f_v}
    f_v(x):= F_{Y|X=x}(v),
    \qquad x\in[0,1].
\end{align}
Since \((X,Y)\sim C\), both \(X\) and \(Y\)
are uniform on \((0,1)\), and thus
\begin{align}\label{eq:fv_mean}
    \int_0^1 f_v(x)\,\de x
    =
    \mathbb{P}(Y\le v)
    =
    v.
\end{align}
By the conditional independence assumption~(ii), we have
\[
    \mathbb{P}(Y\le v\mid U,R,X)
    =
    \mathbb{P}(Y\le v\mid X)
    =
    f_v(X)
    \qquad \mathbb{P}\text{-almost surely.}
\]
Taking conditional expectations with respect to \(U\) yields
\begin{align}
    \mathbb{P}(Y\le v\mid U)
    &=
    \E\!\left[
        \mathbb{P}(Y\le v\mid U,R,X)
        \,\middle|\, U
    \right] \nonumber
    =    \E[f_v(X)\mid U].
\end{align}
Recall that \(X=F_S(U+R)\), where \(R\) is uniform and independent of
\(U\). Consequently, for \(t\in[0,1]\),
\begin{align}\label{eq:conditional_Kfv}
    F_{Y|U=t}(v) = \mathbb{P}(Y\le v\mid U=t)
    &=
    \int_0^1
        f_v\bigl(F_S(t+r)\bigr)\,\de r =
    (Kf_v)(t),
\end{align}
where \(K\) is the smoothing operator in \eqref{eq:K_def}.

Using \eqref{eq:fv_mean}, \(Kf_v\) is \(1\)-Lipschitz by Lemma~\ref{lem:Markov_smoothing}. Hence, Lemma \ref{lem:Lip_cx} gives
\begin{align}\label{eq:proof_ccx_K}
    (Kf_v)(U)
    \leq_{\mathrm{cx}}
    g_v(U).
\end{align}
Combining \eqref{eq:conditional_Kfv},
\eqref{eq:proof_ccx_K}, and the definition of \(g_v\) in \eqref{eq:g_v}, we obtain \(F_{Y|U}(v) \leq_{cx} F_{X|U}(v)\)
for all \(v\in[0,1]\). This proves \eqref{eq_lem:ccx_comparison}.
\end{proof}

The following result shows that the coupling \((U,V,X,Y)\) is suitable in the sense that its transport cost from \(\Pi\) to \(C\) does not exceed the cost of the optimal diagonal transport from \(\Pi\) to \(M\).

\begin{corollary}[Transport-cost comparison]\label{cor_transcost}
    For the coupling \((U,V,X,Y)\) constructed in \eqref{eq_constrX}--\eqref{eq:V_cond_transform}, we have
    \begin{align}\label{comp_VYXR}
        (Y,V)\geq_{sm} (X,R)
    \end{align}
    and thus
    \(\E[(V-Y)^2] \leq \E[(R-X)^2]\).
\end{corollary}

\begin{proof}
    By Lemma \ref{lem_coupling}, \(Y\) and \(V\) are comonotone conditionally on \(U\). Further, by definition of \(X\) in \eqref{eq_constrX}, \(X\) and \(R\) are comonotone conditionally on \(U\). By Proposition \ref{lem:ccx_comparison}, we have \((Y,U)\preccurlyeq_{ccx} (X,U)\). Hence, the characterization of the conditional convex order by the supermodular comparison of conditionally comonotone random vectors in Proposition \ref{lem_2_8} gives
    \begin{align*}
        (Y,V) = (F_{Y|U}^{-1}(V),V) \geq_{sm} (F_{X|U}^{-1}(R),R) = (X,R).
    \end{align*}
    Both equalities hold true by Lemma \ref{lem_coupling}\,\ref{lem_coupling5}.
    Then the statement follows from Lemma \ref{lem:J_lower_orthant}.
\end{proof}

\section{Proofs of Section \ref{sec:intro}}\label{sec:main}

The following three lemmas are needed to establish the equality cases in Relation \eqref{eq:main_intro}.

\begin{lemma}[\(1\)-Lipschitz maps preserving the uniform law]\label{lem:uniform_rigidity}
Let \(h\colon[0,1]\to[0,1]\) be \(1\)-Lipschitz, and assume that \(U\) and \(h(U)\) are uniform on \((0,1)\). Then either \(h(u)=u\) for all \(u\), or \(h(u)=1-u\) for all \(u\).
\end{lemma}

\begin{proof}
Continuity and \(h(U)\sim \cU(0,1)\) imply that the range of \(h\) is the entire interval \([0,1]\). Now, choose \(u_0,u_1\) with \(h(u_0)=0\) and \(h(u_1)=1\). The Lipschitz property of \(h\) yields \(1\le|u_1-u_0|\le1\). Hence, either \(u_0 = 0\) and \(u_1=1\), or \(u_0=1\) and \(u_1=0\). 
In the first case, the Lipschitz property of \(h\) gives \(h(u)\leq u\) and \(h(u)\geq u\), so \(h(u) = u\). In the second case, we obtain \(h(u)=1-u\).
\end{proof}

Recall that \(U,R\sim \cU(0,1)\) are independent, \(X = F_S(U+R)\), and that \(g_v = \mathbb{P}(X\leq v\mid U=\cdot)\) in \eqref{eq:g_v} describes the conditional distribution of \(X\) given \(U\). 

\begin{lemma}[Representation of \(Y| U=t\)]
\label{lem:repYU}
Let \(X\) and \(Y\) be constructed by \eqref{eq_constrX} and \eqref{eq:Y_cond_quantile}.
For \(v\in[0,1]\), define
\begin{align}\label{def_h_v}
    h_v(u):= (K f_v)(u) = \int_0^1 f_v(F_S(u+r))\de r,
\end{align}
for \(f_v = F_{Y|X=\cdot}(v)\) in \eqref{def_f_v}. Then \(h_v\) is \(1\)-Lipschitz and, for all \(v\in [0,1]\),
\begin{align}\label{eq_rep_hv42}
    h_v(t) = F_{Y|U=t}(v) \qquad \text{for } \lambda\text{-almost all } t\in [0,1].
\end{align}
\end{lemma}

\begin{proof}
By the existence of regular conditional distributions, we may choose
versions
\begin{align*}
    f_v(x)=\mathbb P(Y\le v\mid X=x),
    \qquad v\in[0,1],
\end{align*}
such that, for every \(x\in[0,1]\), the map
\(v\mapsto f_v(x)\) is nondecreasing; see, e.g.,
\citet[Theorem~6.3]{Kallenberg-2002}.
Recall that \(K\) is the smoothing operator in \eqref{eq:K_def}.
By Lemma \ref{lem_coupling}\,\ref{lem_coupling3}, \(Y\) and \(U\) are conditionally independent given \(X\). This yields
\begin{align*}
    h_v(u)
    &= \int_0^1 f_v(F_S(u+r))\,\de r 
    = \E[f_v(X)\mid U=u] 
    = \E\!\left[
        \mathbb P(Y\le v\mid X,U)
        \,\middle|\, U=u
      \right] \\
    &= \mathbb P(Y\le v\mid U=u) 
    = F_{Y\mid U=u}(v),
\end{align*}
where we use for the third equality conditional independence of \(Y\) and \(U\) given \(X\) due to Lemma \ref{lem_coupling}\,\ref{lem_coupling3}. The \(1\)-Lipschitz property of \(h_v\) follows from Lemma \ref{lem:Markov_smoothing}.
\end{proof}

Recall that \(h_v\) and \(g_v\) defined in \eqref{def_h_v} and \eqref{eq:g_v} describe the conditional distribution of \(Y\) given \(U\) and \(X\) given \(U\), respectively.

\begin{lemma}[Equality cases]
\label{lem:smoothing_rigidity}
Suppose that
\begin{align}\label{eq:smoothing_equimeasurability}
    h_v(U)\eqd g_v(U)
    \qquad\text{for all }v\in[0,1].
\end{align}
Then either \(Y=X\) \(\mathbb P\)-almost surely or
\(Y=1-X\) \(\mathbb P\)-almost surely.
\end{lemma}

\begin{proof}
The moving-average representation of \(K\) in \eqref{int112} shows that \(Kf\) is absolutely continuous with
\begin{align}\label{eq:K_derivative}
 (Kf)'(u)=f(F_S(u+1))-f(F_S(u))
 \qquad\text{for } \lambda\text{-almost all } u\in(0,1).
\end{align}

At \(v=1/2\), the quantile function of \(S\) in \eqref{eq:q_def} satisfies \(q_S(1/2)=1\), so we obtain from \eqref{eq:g_v} that \(g_{1/2}(u)=1-u\). Hence, by \eqref{eq:smoothing_equimeasurability}, \(h_{1/2}(U)\) is uniform on \((0,1)\), and, by Lemma~\ref{lem:Markov_smoothing}, the function \(h_{1/2}\) is \(1\)-Lipschitz. Lemma~\ref{lem:uniform_rigidity} therefore gives
\begin{align}\label{eq_2hcases}
 h_{1/2}(u)=1-u\quad\text{for all }u,
 \qquad\text{or}\qquad
 h_{1/2}(u)=u\quad\text{for all }u.
\end{align}
We first consider the decreasing case in \eqref{eq_2hcases}. Equation~\eqref{eq:K_derivative} yields
\[
 h_{1/2}'(u) = f_{1/2}(F_S(u+1))-f_{1/2}(F_S(u))=-1
 \quad\text{for } \lambda\text{-almost all }u.
\]
Since \(f_{1/2}\) maps into \([0,1]\), it follows that 
\begin{align*}
    f_{1/2}(F_S(u+1)) = 0 \quad \text{and} \quad f_{1/2}(F_S(u)) = 1.
\end{align*}
Since \(F_S\) is absolutely continuous, strictly increasing, with absolutely continuous inverses, and maps \((0,1)\) and \((1,2)\) to \((0,1/2)\) and \((1/2,1)\), respectively, we obtain
\begin{align}\label{eq:median_threshold}
 f_{1/2}(x)=\1_{\{x<1/2\}}
 \qquad\text{for } \lambda\text{-almost all }x\in [0,1].
\end{align}

Fix \(v<1/2\). Since \(v\mapsto f_v(x)\) is increasing, we have \(f_v\le f_{1/2}\) and thus \(f_v=0\) almost everywhere on \((1/2,1)\). Equation~\eqref{eq:K_derivative} consequently gives
\begin{align}\label{h_vkleiner12}
    h_v'(u)=-f_v(F_S(u))\le0\quad\text{almost everywhere}.
\end{align}
For \(v>1/2\), instead \(f_v\ge f_{1/2}\), and hence \(f_v=1\) almost everywhere on \((0,1/2)\). In this case, we have
\begin{align}\label{h_vnkleiner12}
    h_v'(u)=f_v(F_S(u+1))-1\le0\quad\text{almost everywhere}.
\end{align}
From the first case in \eqref{eq_2hcases} and from \eqref{h_vkleiner12} and \eqref{h_vnkleiner12}, we conclude that \(h_v\) is nonincreasing. Since \(h_v\) and \(g_v\) are nonincreasing functions with the same distribution under Lebesgue measure due to \eqref{eq:smoothing_equimeasurability}, they agree almost everywhere. Since both functions are continuous, it follows in particular that
\begin{align}\label{eq:slice_equality_rigidity}
 h_v(u)=g_v(u)\qquad\text{for all }u,v\in[0,1].
\end{align}

We now aim to recover the conditional distribution functions \(f_v\) in \eqref{def_f_v} from the above identities. For \(v<1/2\), differentiating \(h_v\) and \(g_v\), we obtain from \eqref{h_vkleiner12}, \eqref{eq:g_v}, and  \eqref{eq:slice_equality_rigidity} that
\begin{align*}
 f_v(F_S(u))=\1_{\{u<q_S(v)\}}
 \quad\text{for }\lambda\text{-almost all }u\in(0,1).
\end{align*}
Together with \(f_v=0\) on \((1/2,1)\), this implies 
\begin{align}\label{eq_concl}
    f_v(x)=\1_{\{x<v\}} \qquad \text{for } \lambda\text{-almost all } x\in [0,1].
\end{align}
For \(v>1/2\), differentiation gives similarly
\[
 f_v(F_S(u+1))=1+g_v'(u)
 =\1_{\{u<q_S(v)-1\}}
 \quad\text{for }\lambda\text{-almost all }u\in(0,1).
\]
Combined with \(f_v=1\) on \((0,1/2)\), this gives the same conclusion \eqref{eq_concl}. The case \(v=1/2\) is covered by \eqref{eq:median_threshold}. 
Hence, for each \(v\in [0,1]\), we have \(P(Y\leq v\mid X=x) = \1_{\{x \leq v\}}\) for \(\lambda\)-almost all \(x\in (0,1)\). Since \(Y\) and \(X\) are uniform on \((0,1)\), standard arguments yield \(Y = X\) \(\mathbb{P}\)-almost surely.

Finally, for the second case in \eqref{eq_2hcases}, let \(h_{1/2}(u)=u\) and set
\[
 \widetilde U=1-U,\qquad \widetilde R=1-R,\qquad
 \widetilde X=1-X.
\]
Since \(F_S(2-s)=1-F_S(s)\), we have \(\widetilde X=F_S(\widetilde U+\widetilde R)\). Moreover, \(\mathbb P(Y\le v\mid \widetilde U=u)=h_v(1-u)\), so \(\mathbb P(Y\le v\mid \widetilde U)\eqd g_v(U)\) for every \(v\in[0,1]\). In particular, 
\(\mathbb P(Y\le 1/2\mid \widetilde U=u) =h_{1/2}(1-u)=1-u. \) Thus the first case applies to \((\widetilde U,\widetilde R,\widetilde X,Y)\), and hence \(Y=\widetilde X=1-X\) \(\mathbb{P}\)-almost surely.
\end{proof}

We can now prove our main result. 

\begin{proof}[Proof of Theorem~\ref{thm:main}]
Fix \(C\in\cC\) and consider the coupling
\((U,V,X,Y)\) constructed in
\eqref{eq_constrX}--\eqref{eq:V_cond_transform}.
By Lemma~\ref{lem_coupling}, we have
\[
    (U,V)\sim\Pi
    \qquad\text{and}\qquad
    (X,Y)\sim C.
\]
Hence, the joint law of \((U,V,X,Y)\) defines an admissible
transport coupling from \(\Pi\) to \(C\). Then,
Corollary~\ref{cor_transcost} yields
\begin{align}\label{est_WSES}
\begin{split}
    \mathcal{W}_2^2(\Pi,C)
    &\leq \E\bigl[(U-X)^2]+\E[(V-Y)^2\bigr]\\
    &\leq \E[(U-X)^2]+\E[(R-X)^2]
    =\frac{1}{20}+\frac{1}{20}
     =\frac{1}{10},
\end{split}
\end{align}
where the last equality follows from
Corollary~\ref{cor:half_cost}. This proves
\eqref{eq:main_intro}.

It remains to characterize the equality cases in
\eqref{est_WSES}. By Proposition~\ref{lem:diag_transport},
equality is attained by \(M\) and \(W\).

Conversely, suppose that
\[
    \mathcal{W}_2^2(\Pi,C)=\frac{1}{10}.
\]
Since the transport coupling \((U,V)\mapsto (X,Y)\) has cost at most
\(1/10\), the inequalities in \eqref{est_WSES} must be equalities.
In particular,
\begin{align}\label{eq:cost_equality_main}
    \E[(V-Y)^2]
    =
    \E[(R-X)^2]
    =
    \frac{1}{20}.
\end{align}

By Proposition~\ref{lem:ccx_comparison}, we have
\((Y,U)\preccurlyeq_{\mathrm{ccx}}(X,U)\).
Together with Lemma~\ref{lem_coupling}\,\ref{lem_coupling5} and
Proposition~\ref{lem_2_8}, this yields
\begin{align}\label{eq:sm_equality_main}
    (Y,V)
    =
    (F_{Y|U}^{-1}(V),V)
    \geq_{\mathrm{sm}}
    (F_{X|U}^{-1}(R),R)
    =
    (X,R).
\end{align}
Combining \eqref{eq:cost_equality_main} with
Lemma~\ref{lem:J_lower_orthant}, we obtain
\begin{align}\label{eq:YV_XR}
    (Y,V)\eqd(X,R).
\end{align}
Now fix \(v\in[0,1]\). Then we obtain
\begin{align*}
    \mathbb P(Y\leq v,V\leq w) &= \E\left[\mathbb{P}\bigl(F_{Y|U}^{-1}(V)\leq v,V\leq w \mid U\bigr)\right] 
    = \E\left[\mathbb{P}(V \leq h_v(U),V\leq w \mid U)\right] \\
    &= \E\left[\mathbb{P}(V\leq \min\{h_v(U),w\} \mid U)\right]
    =
    \E\bigl[\min\{h_v(U),w\}\bigr],
\end{align*}
where we use Lemma \ref{lem_coupling}\,\ref{lem_coupling5} for the first equality. The second equality follows with the representation of \(h_v\) in \eqref{eq_rep_hv42} and the identity \(F^{-1}(a)\leq v\) \(\Longleftrightarrow\) \(a\leq F(v)\) for any distribution function \(F\). For the last equality, we use that \(V\) is uniform on \((0,1)\) and independent of \(U\) by Lemma \ref{lem_coupling}\,\ref{lem_coupling1}.\\
Similarly, we obtain
\begin{align*}
    \mathbb P(X\leq v,R\leq w)
    &= \E\!\left[
        \mathbb P\bigl(F_{X|U}^{-1}(R)\leq v,R\leq w\mid U\bigr)
       \right] 
    = \E\!\left[
        \mathbb P\bigl(R\leq g_v(U),R\leq w\mid U\bigr)
       \right] \\
    &= \E\!\left[
        \mathbb P\bigl(R\leq \min\{g_v(U),w\}\mid U\bigr)
       \right] 
    = \E\bigl[\min\{g_v(U),w\}\bigr].
\end{align*}
Therefore, \eqref{eq:YV_XR} implies
\begin{align}\label{eq:min_equality_main}
    \E\bigl[\min\{h_v(U),w\}\bigr]
    =
    \E\bigl[\min\{g_v(U),w\}\bigr]
    \qquad\text{for all }w\in[0,1].
\end{align}
Since \(\E[h_v(U)]
    =
    \mathbb P(Y\leq v)
    =
    v
    =
    \mathbb P(X\leq v)
    =
    \E[g_v(U)],
\)
the stop-loss characterization of the convex order (see e.g. \citet[Theorem~3.A.1]{Shaked-Shanthikumar-2007}) yields
\[
    h_v(U)\eqd g_v(U)
    \qquad\text{for all }v\in[0,1].
\]
Lemma~\ref{lem:smoothing_rigidity} therefore gives
\[
    Y=X
    \qquad\text{or}\qquad
    Y=1-X
    \qquad \mathbb P\text{-almost surely}.
\]
Since \((X,Y)\sim C\), the first case corresponds to \(C=M\), whereas
the second corresponds to \(C=W\). This completes the proof.
\end{proof}

\begin{proof}[Proof of Corollary \ref{cor:main1}.]
By Theorem~\ref{thm:main}, we have \(0\leq \mathcal W_2(C,\Pi)\leq \frac{1}{\sqrt{10}},\)
which proves (i). Moreover, since \(\mathcal W_2\) is a metric,
\[
    \mathfrak D_{\mathcal W}(X,Y)=0
    \quad\Longleftrightarrow\quad
    C=\Pi.
\]
Since \(X\) and \(Y\) have continuous marginals, Sklar's theorem \cite[Theorem~2.3.3]{Nelsen-2006} yields
\(C=\Pi\) if and only if \(X\) and \(Y\) are independent, proving (ii).

Finally, Theorem~\ref{thm:main} gives
\[
    \mathfrak D_{\mathcal W}(X,Y)=1
    \quad\Longleftrightarrow\quad
    C\in\{M,W\}.
\]
For continuous marginals, \(C=M\) is equivalent to comonotonicity of
\(X\) and \(Y\), whereas \(C=W\) is equivalent to countermonotonicity.
This proves (iii).
\end{proof}

\begin{proof}[Proof of Corollary~\ref{cor_main}]
By Proposition~\ref{lem:diag_transport}, the map
\[
    T(u,r)=\bigl(F_S(u+r),F_S(u+r)\bigr)
\]
is an optimal Monge transport from \(\Pi\) to \(M\). It remains to verify that \(T\) is the gradient of the convex potential in
\eqref{eq:Brenier_potential}.

Writing \(\Phi(u,r)=\psi(u+r)\), direct differentiation of
\eqref{eq:Brenier_potential} gives
\(\psi'(s)=F_S(s)\) for all \(s\in[0,2]\),
where \(F_S\) is given by \eqref{eq:FS}. Hence
\[
    \nabla\Phi(u,r)
    =
    \bigl(\psi'(u+r),\psi'(u+r)\bigr)
    =
    \bigl(F_S(u+r),F_S(u+r)\bigr)
    =
    T(u,r).
\]
Since \(F_S\) is nondecreasing, \(\psi\) is convex, and therefore
\(\Phi(u,r)=\psi(u+r)\) is convex on \([0,1]^2\). Thus \(T=\nabla\Phi\)
is the Brenier map from \(\Pi\) to \(M\).
\end{proof}


\section*{Acknowledgements}

The author thanks Stefan Schrott for presenting this problem as an open question at the DMV Annual Meeting 2026 in Konstanz and thereby bringing it to his attention.
For assistance with exposition and technical checks, a large language model was used. The proof strategy is due to the author, who independently verified all mathematical arguments. This research was funded in whole by the Austrian Science Fund (FWF) {[10.55776/PAT1669224]} project \emph{Stochastic orders for functional dependence}.

\bibliographystyle{plainnat}
\bibliography{Max_Wasserstein_Distance}

\end{document}